\documentclass{article}

\usepackage{
 amsmath,
 amsxtra,
 amsthm,
 amssymb,
 etex,
 mathrsfs,
 }
\usepackage[all]{xy}

\newtheorem{theorem}{Theorem}[section]
\newtheorem{lemma}[theorem]{Lemma}
\newtheorem{conjecture}[theorem]{Conjecture}
\newtheorem{proposition}[theorem]{Proposition}
\newtheorem{corollary}[theorem]{Corollary}

\theoremstyle{definition}
\newtheorem{defn}[theorem]{Definition}
\newtheorem{remark}[theorem]{Remark}

\newcommand{\bd}{\begin{defn}}
\newcommand{\ed}{\end{defn}}
\newcommand{\bl}{\begin{lemma}}
\newcommand{\el}{\end{lemma}}
\newcommand{\bp}{\begin{proposition}}
\newcommand{\ep}{\end{proposition}}
\newcommand{\bt}{\begin{theorem}}
\newcommand{\et}{\end{theorem}}
\newcommand{\bc}{\begin{corollary}}
\newcommand{\ec}{\end{corollary}}
\newcommand{\br}{\begin{remark}}
\newcommand{\er}{\end{remark}}
\newcommand{\ba}{\begin{array}}
\newcommand{\ea}{\end{array}}
\newcommand{\bpf}{\begin{proof}}
\newcommand{\epf}{\end{proof}}

\newcommand{\Z}{\mathbb{Z}}
\newcommand{\Q}{\mathbb{Q}}
\newcommand{\Zp}{\mathbb{Z}_{p}}
\newcommand{\Qp}{\mathbb{Q}_{p}}

\newcommand{\Ep}{E[p^{\infty}]}

\newcommand{\Ga}{\Gamma}
\newcommand{\ga}{\gamma}

\newcommand{\La}{\Lambda}

\newcommand{\si}{\sigma}
\newcommand{\Si}{\Sigma}
\newcommand{\ze}{\zeta}

\newcommand{\s}{\overrightarrow{s}}

\DeclareMathOperator{\Sel}{Sel} \DeclareMathOperator{\Gal}{Gal}

\newcommand{\cyc}{\mathrm{cyc}}

\newcommand{\Sss}{\Sigma^{ss}}

\newcommand{\ot}{\otimes}
\newcommand{\ilim}{\displaystyle \mathop{\varinjlim}\limits}
\newcommand{\plim}{\displaystyle \mathop{\varprojlim}\limits}
\newcommand{\im}{\mathrm{im}\,}

\newcommand{\lra}{\longrightarrow}

\newcommand{\ps}[1]{[[ #1 ]]}

  \DeclareFontFamily{U}{wncy}{}
  \DeclareFontShape{U}{wncy}{m}{n}{<->wncyr10}{}
  \DeclareSymbolFont{mcy}{U}{wncy}{m}{n}
  \DeclareMathSymbol{\sha}{\mathord}{mcy}{"58}

\begin{document}
\title{On the algebraic functional equation for the mixed signed Selmer group over multiple $\Zp$-extensions}
 \author{ Suman Ahmed\footnote{School of Mathematics and Statistics,
Central China Normal University, Wuhan, 430079, P.R.China.
 E-mail: \texttt{npur.suman@gmail.com}}  \quad
  Meng Fai Lim\footnote{School of Mathematics and Statistics $\&$ Hubei Key Laboratory of Mathematical Sciences,
Central China Normal University, Wuhan, 430079, P.R.China.
 E-mail: \texttt{limmf@mail.ccnu.edu.cn}} }
\date{}
\maketitle

\begin{abstract} \footnotesize
\noindent
Let $E$ be an elliptic curve defined over a number field with good reduction at all primes above a fixed odd prime $p$, where at least one of which is a supersingular prime of $E$. In this paper, we will establish the algebraic functional equation for the mixed signed Selmer groups of an elliptic curve with good reduction at every prime above $p$ over a multiple $\Zp$-extension.


\end{abstract}

\section{Introduction}

Throughout this article, $p$ will always denote an odd prime number. Let $E$ be an elliptic curve defined over a number field $F$. If $E$ has good ordinary reduction at every prime of $F$ above $p$, a well-known conjecture of Mazur \cite{Maz} asserts that the $p$-primary Selmer group of $E$ over the cyclotomic $\Zp$-extension is cotorsion over
$\Zp\ps{\Gal(F^\cyc/F)}$. Following Iwasawa \cite{Iw73}, Mazur went on further predicting that the characteristic ideal of this cotorsion Selmer group has a precise description in terms of an appropriate $p$-adic $L$-function which nowadays is coined the Iwasawa main conjecture (also see \cite{G89, K, SU}). In view of the above correspondence, one would expect that there should be an algebraic relation of the characteristic ideals of the Selmer groups mirroring the functional equation of the $p$-adic $L$-function. Indeed, this has been extensively studied by Greenberg \cite{G89}, where he even established such a relation (which he coined as ``the algebraic functional equation") unconditionally without assuming the main conjecture nor the existence of the $p$-adic $L$-function. Subsequently, this direction of study has been generalized to more general $p$-adic Lie extensions (for instance, see \cite{LLTT, Za08, Za10}).

However, if the elliptic curve $E$ has supersingular reduction at one prime above $p$, then the $p$-primary Selmer group of $E$ over $F^\cyc$
is not expected to be cotorsion over $\Zp\ps{\Gal(F^\cyc/F)}$ (see \cite{Sch85}). It was due to the remarkable insight of Kobayashi \cite{Kob} that one has to replace the consideration of the $p$-primary Selmer group by some smaller subgroups. More precisely, Kobayashi constructed the plus and minus Selmer groups of an elliptic curve over $\Q^\cyc$ which are contained in the $p$-primary Selmer groups. He was able to show that these signed Selmer groups are cotorsion over $\Zp\ps{\Gal(\Q^\cyc/\Q)}$, and formulate a main conjecture (and even prove one-side divisibility) relating the characteristic ideals of these signed Selmer groups to the signed $p$-adic $L$-functions of Pollack \cite{Po}. Motivated by this main conjecture, Kim \cite{Kim08} first established an algebraic functional equation for the signed Selmer group over a cyclotomic $\Zp$-extension (also see \cite{AL3, LeiPon, Sp17}). In \cite{Kim14}, Kim developed a theory of mixed signed Selmer groups of an elliptic curve over $\Zp^2$-extension of an imaginary quadratic field, where he chose either the plus or minus condition for each supersingular prime, and so in this situation, one has four signed Selmer groups to work with. He then proposed a main conjecture relating these Selmer groups to the mixed signed $p$-adic $L$-functions of Loeffler \cite{Loe}. Recently, B\"{u}y\"{u}kboduk-Lei \cite{BL3} has established an analytic functional equation for (some of) the $p$-adic $L$-functions of Loeffler.
In view of this and the main conjecture, one would expect that the signed Selmer groups would satisfy an algebraic functional equation.

The goal of this paper is to prove such an algebraic functional equation. In fact, we shall put ourselves in a slightly more general framework following that of Kitajima-Otsuki \cite{KO}. Namely, our elliptic curve has good reduction at all primes of $F$ above $p$, where at least one of which is a supersingular prime of the said elliptic curve. Let $F_\infty$ be a $\Zp^d$-extension of $F$ which contains the cyclotomic $\Zp$-extension $F^{\cyc}$ and satisfies the property that every prime of $F^{\cyc}$ above $p$ at which $E$ has good supersingular reduction is unramified in $F_\infty/F^{\cyc}$. The latter ramification condition on $F_\infty$ is necessary for us to be able to apply the work of Kim \cite{Kim14} in defining the plus/minus norm groups which in turn is required for the definition of the mixed signed Selmer groups over the $\Zp^d$-extension $F_\infty$. We should mention that the mixed signed Selmer group defined here is a blend of those considered in \cite{Kim14} and \cite{KO}, which was also studied in \cite{LLAk}. In \cite{LLTT}, Lai-Longhi-Tan-Trihan developed a theory of $\Ga$-system (where $\Ga\cong \Zp^d$) which they utilized to prove an algebraic functional equation for abelian varieties with good ordinary reduction at all primes above $p$ over a $\Zp^d$-extension of a global field. We shall adopt their approach in proving our algebraic functional equation for the mixed signed Selmer group over the $\Zp^d$-extension $F_\infty$. To the best knowledge of the authors, the (appropriate) $p$-adic $L$-functions have not been constructed in this mixed reduction setting. Therefore, besides providing evidence to the main conjecture in the scenario considered by Kim and Loeffler, our result also provides some positive belief for the conjectural functional equation of the (conjectural) $p$-adic $L$-functions in this mixed reduction setting.

It would definitely be interesting to study whether the techniques in this article can be applied to the setting of the multi-signed Selmer groups for a non-ordinary
motive with Hodge-Tate weights being $0$ and $1$ as considered by B\"{u}y\"{u}kboduk-Lei \cite{BL, BL2}. Over the cyclotomic $\Zp$-extension, such an algebraic functional equation for these multi-signed Selmer groups has been established in \cite{LeiPon}. However, unlike the elliptic curve situation, where a local theory of plus/minus norm groups over the $\Zp^2$-extension of $\Qp$ is available thanks to the work of Kim \cite{Kim14}, it is not clear to the authors how to extend the local theory of B\"{u}y\"{u}kboduk-Lei to the $\Zp^2$-extension of $\Qp$ at this point of writing.

We now give an outline of the paper. In Section \ref{elliptic curve over local}, we review the plus/minus conditions of Kobayashi as extended by Kim to a $\Zp^2$-extension of $\Qp$. We also discuss a result of Kim on the orthogonality properties of these plus/minus conditions under the local Tate pairing. Section \ref{Selmer} is where we define the signed Selmer groups. We also introduce the strict signed Selmer groups. These latter Selmer groups coincide with the signed Selmer group on the level of $F_\infty$ (but may differ on the intermediate finite subextensions). However, the local conditions in defining these strict signed Selmer group satisfy certain orthogonality conditions, which is an important aspect for our proof, as it allows us to invoke a theorem of Flach \cite{Fl}. Our main theorem (Theorem \ref{main theorem}) is stated in this section, and its proof will be given in Section \ref{Proof sec}. In this last section, we will show that the strict signed Selmer groups, via a combination of the results of Kim and Flach, give rise to a $\Ga$-system in the sense of Lai-Longhi-Tan-Trihan \cite{LLTT}. From there, we apply the machinery developed by Lai \textit{et al} to establish our theorem.

\subsection*{Acknowledgement}
We thank Byoung Du Kim and Antonio Lei for their interests and comments. We also like to thank Antonio Lei for making us aware of the paper \cite{BL3}.
The research of this article took place when S. Ahmed was a postdoctoral fellow at Central China Normal University, and he would like to acknowledge the hospitality
and conducive working conditions provided by the said institute.
M. F. Lim is supported by the National Natural Science Foundation of China under Grant No. 11550110172 and Grant No. 11771164.

\section{Supersingular elliptic curves over local fields} \label{elliptic curve over local}

In this section, we record certain results on supersingular elliptic curves over a $p$-adic local field. Let $E$ be an elliptic curve defined over $\Qp$ with good supersingular reduction and $a_p=1 + p - |\widetilde{E}(\mathbb{F}_p)| = 0$, where $\widetilde{E}$ is the reduction of $E$. Denote by $\widehat{E}$ the formal group of $E$. For convenience, if $L$ is an extension of $\Qp$, we shall write $\widehat{E}(L)$ for $\widehat{E}(\mathfrak{m}_L)$, where $\mathfrak{m}_L$ is the maximal ideal of the ring of integers of $L$. Fix a finite unramified extension $K$ of $\Qp$. Denote by $K^{\cyc}$ (resp., $K^{nr}$) the cyclotomic (resp, the unramified) $\Zp$-extension  of $K$. If $n\ge0$ is an integer, we write  $K_n$  (resp. $K^{(n)}$) for  the unique subextension of $K^{\cyc}/K$ (resp. $K^{nr}/K$), whose degree over $K$ is equal to $p^n$.

\begin{lemma} \label{supersingular points}
 The formal groups $\widehat{E}(K^{(m)}K_n)$ has no $p$-torsion for all integers $m, n\ge0$. In particular, $E(K^{(m)}K_n)$ has no $p$-torsion for every $m, n$.
 \end{lemma}

\begin{proof}
The first assertion is \cite[Proposition 3.1]{KO} or \cite[Proposition 8.7]{Kob}. For the second assertion, consider the following short exact sequence
\[ 0\lra \widehat{E}(K^{(m)}K_n)\lra E(K^{(m)}K_n) \lra \widetilde{E}(k_{m,n})\lra 0,\]
where $k_{m,n}$ is the residue field of $K^{(m)}K_n$. Since $\widetilde{E}(k_{m,n})$ has no $p$-torsion by our assumption that $E$ has good supersingular reduction, the second assertion follows from the first assertion.
\end{proof}

Following \cite{Kim07,Kim08,Kim14,KO,Kob,LLAk}, we define
the following plus/minus norm groups.

\begin{defn} \label{plus minus defn}
\[\widehat{E}^+(K^{(m)}K_n) = \left\{ P\in \widehat{E}(K^{(m)}K_n)~:~\mathrm{tr}_{n/l+1}(P)\in E(K^{(m)}K_{l}), 2\mid l, 0\leq l \leq n-1\right\}, \]
\[\widehat{E}^-(K^{(m)}K_n) = \left\{ P\in \widehat{E}(K^{(m)}K_n)~:~\mathrm{tr}_{n/l+1}(P)\in E(K^{(m)}K_{l}), 2\nmid l, 0\leq l \leq n-1\right\}, \]
where $\mathrm{tr}_{n/l+1}: \widehat{E}(K^{(m)}K_n) \lra \widehat{E}(K^{(m)}K_{l+1})$ denotes the trace map with respect to the formal group law of $\widehat{E}$.
\end{defn}

For the rest of this subsection, we write $L_{\infty} = \cup_{m,n\ge0}K^{(m)}K_n$. Note that $G:=\Gal(L_{\infty}/K)\cong \Zp^2$. Denote by $G_n$ the subgroup $G^{p^n}$ of $G$ and set $L_n$  to be the fixed field of $G_n$. In particular, one has $\Gal(L_n/K)\cong \Z/p^n\times \Z/p^n$.

Write $\mathbb{H}^{\pm}_{\infty}= \hat{E}^{\pm}(L_{\infty})\ot\Qp/\Zp$ and $\mathbb{H}^{\pm}_n=\big(\mathbb{H}^{\pm}_{\infty}\big)^{G_n}$. By \cite[Lemma 8.17]{Kob}, the group $\mathbb{H}^{\pm}_{\infty}$ injects into $H^1(L_\infty,\Ep)$ via the Kummer map. In view of Lemma \ref{supersingular points}, it follows from the Hochshild-Serre spectral sequence that there is an isomorphism
\[ H^1(L_n,\Ep)\cong H^1(L_{\infty},\Ep)^{G_n}.\]
Thus, $\mathbb{H}^{\pm}_n$ can be viewed as a subgroup of $H^1(L_n,\Ep)$ via this isomorphism. By Lemma \ref{supersingular points} again, we have an identification
\[ H^1(L_n,E[p^m])\cong H^1(L_n,\Ep)[p^m]\]
which further enables us to view $\mathbb{H}^{\pm}_n[p^m]$ as a subgroup of $H^1(L_n,E[p^m])$. Under these observations, we may now state the following result of Kim.

\bp[Kim] \label{Kim duality}
The group $\mathbb{H}^{-}_n[p^m]$ is the exact annihilator of itself with respect to the local Tate pairing
\[ H^1(L_n, E[p^m])\times H^1(L_n, E[p^m])\lra \Z/p^m. \]
If $|K:\Qp|$ is not divisible by $4$, we also have the same conclusion for $\mathbb{H}^{+}_n[p^m]$.
\ep

\bpf
See \cite[Proposition 3.15]{Kim07} and \cite[Theorem 2.9]{Kim14}.
\epf

\section{Signed Selmer groups} \label{Selmer}

For the remainder of the paper, $E$ will denote an elliptic curve defined over a number field $F'$. Let $F$ be a finite extension of $F'$.
We introduce the following axiomatic conditions which will be in force throughout the paper.

\begin{itemize}
\item[\textbf{(S1)}]  The elliptic curve $E$ has good reduction at all primes of $F'$ above $p$, and at least one of which is a supersingular reduction prime of $E$.

 \item[\textbf{(S2)}] For each prime $u$ of $F'$ above $p$ at which $E$ has good supersingular reduction, we always have

 (a) $F'_u\cong\Qp$ and $u$ is unramified in $F/F'$;

 (b) $a_u = 1 + p - |\widetilde{E}_u(\mathbb{F}_p)| = 0$, where $\widetilde{E}_u$ is the reduction of $E$ at $u$.
\end{itemize}

From now on, we fix a finite set $\Si$ of primes of $F$ which contains all the primes above $p$, all the ramified primes of $F/F'$, the bad reduction primes of $E$ and the archimedean primes. Let $F_\Si$ denote the maximal algebraic extension of $F$ which is unramified outside $\Si$. For every extension $L$ of $F$ contained in $F_\Si$, we write $G_\Si(L)=\Gal(F_\Si/L)$. Denote by $\Sss$ the set of primes of $F$ above $p$ at which $E$ has good supersingular reduction. By (S1), the set $\Sss$ is nonempty. For any subset $S$ of $\Si$ and any extension $L$ of $F$, we shall write $S(L)$ for the set of primes of $L$ above $S$.

Let $F_{\infty}$ be a $\Zp^d$-extension of $F$ which always satisfies the following hypothesis.

\begin{itemize}
\item[\textbf{(S3)}] The field $F_{\infty}$ contains the cyclotomic $\Zp$-extension $F^{\cyc}$, and has the property that every $w\in\Sss(F^{\cyc})$ is unramified in $F_{\infty}/F^{\cyc}$.
\end{itemize}

Note that $F_\infty\subseteq F_\Si$ (cf.\ \cite[Theorem 1]{Iw73}). Write $\Ga=\Gal(F_{\infty}/F)$ ($\cong \Zp^d$). Let $F_n$ be the unique subextension of $F_{\infty}/F$ such that $\Gal(F_n/F)\cong (\Z/p^n)^d$. Let $\s=(s_v)_{v\in \Sss}\in\{+,-\}^{\Sss}$ be a fixed choice of signs. For each $w\in \Sss(F_n)$, we set $s_w = s_v$, where $v$ is the prime of $F$ below $w$. By (S3), $F_{n,w}$ is the compositum of a subextension of the cyclotomic $\Zp$-extension of $F_v$ and a subextension of the unramified $\Zp$-extension of $F_v$. Hence we can define $\widehat{E}^{s_w}(F_{n,w})$ as in Definition \ref{plus minus defn}.

\begin{defn}
For $\s=(s_v)_{v\in \Sss}\in\{+,-\}^{\Sss}$, the signed Selmer group $\Sel^{\s}(E/F_n)$ is then defined by
\[ \ker \left(H^1(G_\Sigma(F_n),\Ep)\lra \bigoplus_{w\in \Sss(F_n)}\frac{H^1(F_{n,w},\Ep)}{\widehat{E}^{s_w}(F_{n,w})\otimes\Qp/\Zp}\times\bigoplus_{w\in \Sigma(F_n)\setminus \Sss(F_n)}\frac{H^1(F_{n,w},\Ep)}{E(F_{n,w})\otimes\Qp/\Zp} \right).
\]
We set $\Sel^{\s}(E/F_{\infty}) = \ilim_n\Sel^{\s}(E/F_n)$ and write $X^{\s}(E/F_\infty)$ for its Pontryagin dual.
\end{defn}

The module $X^{\s}(E/F_\infty)$ is finitely generated over $\Zp\ps{\Ga}$. One generally expects a stronger assertion on the structure of this module.

\begin{conjecture}\label{conj}
For all choices of $\s\in\{+,-\}^{\Sss}$, the Selmer group $X^{\s}(E/F_{\infty})$ is torsion over $\Zp\ps{\Ga}$.
\end{conjecture}

Over the cyclotomic $\Zp$-extension, when $E$ has good ordinary reduction at all primes above $p$, the above conjecture is precisely Mazur's conjecture in \cite{Maz}, which is known to hold in the case when $E$ is defined over $\Q$ and $F$ an abelian extension of $\Q$ (see \cite{K}). For an elliptic curve over $\Q$ with good supersingular reduction at $p$, this conjecture has been proven to be true by Kobayashi \cite{Kob};  see also \cite{BL, BL2} for a generalization of this conjecture for abelian varieties. Over a multiple $\Zp$-extension,
the above conjecture is a natural extension of Mazur's conjecture (for instance, see \cite{OcV03}). When $E$ has supersingular reduction, this was studied in \cite{KimPark, LLAk, LeiP, LeiS}.

We introduce one last hypothesis.

\begin{itemize}
\item[\textbf{(S4)}] For our fixed choice of $\s$, we have $4 \nmid |F_v :\Qp|$ whenever $s_v =+$.
\end{itemize}

For a $\Zp\ps{\Ga}$-module $M$, we write $M^\iota$ for the $\La$-module which is $M$ as $\Zp$-module with a $\Ga$-action given by $\ga\cdot_{\iota} x = \ga^{-1}x$ for $\ga\in\Ga$ and $x\in M$. If $M$ and $N$ are two torsion $\Zp\ps{\Ga}$-modules which are pseudo-isomorphic to each other, we write $M\sim N$.

We can now state our main result.

\bt \label{main theorem}
 Suppose that $(S1)-(S4)$ are valid. Assume that $X^{\s}(E/F_\infty)$ is torsion over $\Zp\ps{\Ga}$ for our fixed choice of $\s$. Then we have a pseudo-isomorphism
\[ X^{\s}(E/F_\infty) \sim X^{\s}(E/F_\infty)^{\iota}.\]
In particular, the characteristic ideal of $X^{\s}(E/F_\infty)$ coincides with the characteristic ideal of $X^{\s}(E/F_\infty)^{\iota}$.
\et

The proof of the theorem will be given in Section \ref{Proof sec}. As mentioned in the introduction, in proving our main result, we need to work with the so-called strict signed Selmer groups which we shall recall here. The strict signed Selmer group of $E$ over $F_n$ is given by
\[
    \Sel^{\overrightarrow{s},str}(E/F_n):=\ker\Bigg(H^1(G_{\Sigma}(F_n),\Ep)
    \lra \bigoplus_{w\in\Si(F_n)}\frac{H^1(F_{n,w},\Ep)}
    {C_w\Big(\Ep/F_n\Big)}\Bigg),\]
where $C_w\Big(\Ep/F_n\Big)$ is defined as follow.

\begin{itemize}
\item[$(1)$] Suppose that $w$ divides $p$ and is a good supersingular prime of $E$. We take
    \[C_w\Big(\Ep/F_n\Big):= \Big(\bigoplus_{x|w}E^{s_x}(F_{\infty,x})\ot\Qp/\Zp\Big)^{\Ga_n}, \]
    where $x$ runs through all the primes of $F_{\infty}$ above $w$.

\item[$(2)$] If $w$ divides $p$ and is a good ordinary reduction prime of $E$, we set $C_w\Big(\Ep/F_n\Big)$ to be
\[\Bigg(\im\left(H^1(F_{n,w},\widehat{E}[p^{\infty}])\lra H^1(F_{n,w}, \Ep)\right)\Bigg)_{div}.\]

\item[$(3)$] If $w$ does not divide $p$, then $C_w\Big(\Ep/F_n\Big)$ is simply defined to be
the kernel of the map
\[ H^1(F_{n,w}, \Ep) \lra H^1(F_{n,w}^{ur}, \Ep),\]
where $F_{n,w}^{ur}$ is the maximal unramified extension of $F_{n,w}$.
\end{itemize}

Define $\Sel^{\overrightarrow{s},str}(E/F_\infty):=\ilim_n\Sel^{\overrightarrow{s},str}(E/F_n)$.
Note that in the case when $\Sss=\emptyset$, this is the strict Selmer group as defined in the sense of Greenberg \cite{G89}. Our choice of naming the above as the strict signed Selmer group is inspired by this observation.
In general, the strict signed Selmer group needs not agree with the signed Selmer group on the intermediate level $F_n$. But they do coincide upon taking limit and we record this fact here.

\bp \label{strict = usual}
 We have an identification $\Sel^{\overrightarrow{s},str}(E/F_\infty) \cong \Sel^{\overrightarrow{s}}(E/F_\infty)$.
\ep

\bpf
 This has a similar proof to that in \cite[Proposition 4.3]{AL3}.
\epf

Here, and subsequently, we shall write $S_{div}(E/F_n)$ for the $p$-divisible part of $\Sel^{\s,str}(E/F_n)$. In other words, $S_{div}(E/F_n)$ is the maximal $p$-divisible subgroup of $\Sel^{\s,str}(E/F_n)$. The main reason for working with the strict signed Selmer groups lies in the following proposition which will play a crucial role in the proof of Theorem \ref{main theorem}.

\bp \label{Flach}
Suppose that $(S1)-(S4)$ are valid.
There is a perfect pairing
\[ \Sel^{\s,str}(E/F_n)/S_{div}(E/F_n)\times \Sel^{\s,str}(E/F_n)/S_{div}(E/F_n) \lra \Qp/\Zp.\]
\ep

\bpf
By virtue of our assumption $(S4)$, we may apply Proposition \ref{Kim duality} to conclude that the local condition at each supersingular prime is its own annihilator under the Tate pairing. For the local conditions at the remaining primes, such properties are well-known (for instance, see \cite{Guo} or \cite{AL3}). Hence we may apply the main result of Flach \cite{Fl} to obtain the perfect pairing in the proposition.
\epf

We end the section with one more useful observation.

\bl \label{control kernel}
Suppose that $(S1)-(S3)$ are valid.
 For every $n$, the restriction map
\[\Sel^{\overrightarrow{s},str}(E/F_n) \lra\Sel^{\overrightarrow{s}}(E/F_\infty)^{\Ga_n}\]
is an injection, where $\Ga_n = \Gal(F_\infty/F_n)$.
\el

\bpf
A standard diagram chasing argument shows that the kernel of the restriction map is contained in the kernel of the following restriction map
\[H^1(G_\Si(F_n),\Ep)\lra H^1(G_\Si(F_\infty),\Ep)^{\Ga_n}\]
on cohomology, which is precisely given by $H^1(\Ga_n, E(F_{\infty})[p^\infty])$. On the other hand, by Lemma \ref{supersingular points}, we have $E(F_{\infty,w})[p^\infty]=0$ for each $w\in \Sss(F_\infty)$. It follows from this that $E(F_{\infty})[p^\infty]=0$ and so the kernel is trivial which is what we want to show.
\epf

\section{Proof of the main theorem} \label{Proof sec}

As before, $\Ga = \Gal(F_{\infty}/F)\cong\Zp^d$ for some $d\geq 2$. We shall write $\La$ for the Iwasawa algebra $\Zp\ps{\Ga}$. For a finitely generated torsion $\La$-module $M$, we have a pseudo-isomorphism
\[ \bigoplus_i \La/\xi_i^{r_i} \lra M,\]
where $\xi_i$ are irreducible elements in $\La$ (cf. \cite[\S 5]{NSW}). We shall write
\[ [M] = \bigoplus_i \La/\xi_i^{r_i}.\]
Note that the module $[M]$ is uniquely determined by $M$ (up
to isomorphism).

Recall that for a $\Zp\ps{\Ga}$-module $M$, we write $M^\iota$ for the $\La$-module which is $M$ as $\Zp$-module with a $\Ga$-action given by $\ga\cdot_{\iota} x = \ga^{-1}x$ for $\ga\in\Ga$ and $x\in M$. Then one easily sees that
\[ [M^{\iota}] = [M]^{\iota}. \]

In view of this, Theorem \ref{main theorem} is equivalent to saying that
\[ [X^{\s}(E/F_\infty)]^{\iota} = [X^{\s}(E/F_\infty)]. \]

\subsection{First reduction} \label{first reduction}

Write $\mu_{p^{\infty}}$ for the group of all $p$-power roots of unity.
Following \cite{LLTT}, we say that $f \in\La$ is a simple element if there exist $\ga\in \Ga-\Ga^p$ and $\ze \in \mu_{p^{\infty}}$ such
that $f = f_{\ga, \ze}$, where
\[f_{\ga,\ze} := \prod_{\si\in \Gal(\Qp(\mu_p)/\Qp)}\big(\ga- \si(\ze)\big).\]

Let $M$ be a finitely generated $\La$-torsion module. We define $[M]_{si}$ to be the
the sum over those $\xi_i$ in $[M]$ which are simple and $[M]_{ns}$ as its complement. In particular, we have $[M] = [M]_{si}\oplus [M]_{ns}$.
The simple component enjoys the following property.

\bl \label{First observation}
 Let $M$ be a finitely generated torsion $\La$-module.
 Then we have
 \[ [M]^{\iota}_{si} = [M]_{si}.\]
\el

\bpf
See \cite[pp 1930, Formula (9)]{LLTT}.
\epf

In view of the preceding lemma and the remark before Subsection \ref{first reduction}, the proof of Theorem \ref{main theorem} is reduced to showing that the non-simple component of $X^{\s}(E/F_\infty)$ is invariant under $\iota$, or in other words,
\[ [X^{\s}(E/F_\infty)]^{\iota}_{ns} = [X^{\s}(E/F_\infty)]_{ns}. \]

To verify this latter relation, we make use of the theory of $\Ga$-system as developed by Lai-Longhi-Tan-Trihan in \cite{LLTT} which we now recall.

\subsection{$\Ga$-system} \label{Ga-system subsec}

Here, we will only collect the properties of a $\Ga$-system required for the proof of our theorem. For a more in-depth discussion, we refer readers to the paper \cite{LLTT}; see especially Section 3 of the said paper.

\bd
Consider a collection
\[ \mathfrak{A} = \big\{ \mathfrak{a}_n, \mathfrak{b}_n, \langle~,~\rangle_n, r^n_m, c^n_m ~\big|~ n\geq m \geq 0 \big\},\]
which satisfies all of the following properties.

\begin{itemize}
\item[($\Ga$-1)] For every $n$, $\mathfrak{a}_n$ and $\mathfrak{b}_n$ are finite $\Zp$-modules with an action of $\Ga$ factoring through $\Zp[\Ga/\Ga^{p^n}]$.

\item[($\Ga$-2)] For every $n\geq m$, the maps
\[ \ba{c}
 r^n_m: \mathfrak{a}_m\times \mathfrak{b}_m \lra \mathfrak{a}_n\times \mathfrak{b}_n \\
 c^n_m: \mathfrak{a}_n\times \mathfrak{b}_n \lra \mathfrak{a}_m\times \mathfrak{b}_m
\ea\] are $\La$-morphisms such that $r^n_m(\mathfrak{a}_m)\subseteq \mathfrak{a}_n$, $r^n_m(\mathfrak{b}_m)\subseteq \mathfrak{b}_n$, $c^n_m(\mathfrak{a}_n)\subseteq \mathfrak{a}_m$  and $c^n_m(\mathfrak{b}_n)\subseteq \mathfrak{b}_m$ with $r^n_n= c^n_n = \mathrm{id}$. Also, $\{\mathfrak{a}_n\times \mathfrak{b}_n, r^n_m\}_n$ forms an inductive system and $\{\mathfrak{a}_n\times \mathfrak{b}_n, c^n_m\}_n$ forms a projective system.

\item[($\Ga$-3)] The composition $r^n_m\circ c^n_m$ coincides with the norm map associated with $\Gal(F_n/F_m)$ and the composition $c^n_m\circ r^n_m$ coincides with multiplication by $p^{n-m}$.

\item[($\Ga$-4)] For each $n$, $\langle~,~\rangle_n: \mathfrak{a}_n\times\mathfrak{b}_n \lra \Qp/\Zp$ is a perfect pairing which is compatible with the $\Ga$-action and the maps $r^n_m$ and $c^n_m$. More precisely, we have
    \[ \langle \ga a, \ga b\rangle_n = \langle a, b\rangle_n ~\mbox{ for every }\ga\in \Ga, \]
 \[ \langle a_n, r^n_m(b_m)\rangle_n = \langle c^n_m(a_n), b_m\rangle_m ~\mbox{ for every } a_n\in \mathfrak{a}_n, b_m\in\mathfrak{b}_m, \]
 and
  \[ \langle  r^n_m(a_m), b_n\rangle_n = \langle a_m, c^n_m(b_n)\rangle_m ~\mbox{ for every }a_m\in \mathfrak{a}_m, b_n\in\mathfrak{b}_n. \]
\end{itemize}
Write $\mathfrak{a}= \plim_n\mathfrak{a}_n$ and $\mathfrak{b}= \plim_n\mathfrak{b}_n$, where the transition maps are given by $c^n_m$. We then say that $\mathfrak{A}$ is a \textbf{$\Ga$-system} if both
$\mathfrak{a}$ and $\mathfrak{b}$ are finitely generated torsion $\La$-modules.
\ed

\br
We say a little on comparing our notation with that in \cite{LLTT}. Our $r^n_m$ here is $\mathfrak{r}^n_m$ there and $c^n_m$ here is $\mathfrak{k}^n_m$ there. The reason behind our choice of notation here is because for our application, our maps $r^n_m$ (resp. $c^n_m$) are induced by restriction maps on cohomology (resp., corestriction maps on cohomology).
\er

Let $\mathfrak{A}$ be a $\Ga$-system. By ($\Ga$-2), we can define the inductive limit $\ilim_m\mathfrak{a}_m$, whose transition maps are given by $r^n_m$. Note that by ($\Ga$-4), the Pontryagin dual of this inductive limit coincides with $\mathfrak{b}$.

Now, denote by $r_n$ the natural morphism $\mathfrak{a}_n \lra \ilim_m\mathfrak{a}_m$, whose kernel is in turn denoted by $\mathfrak{a}_n^0$. The module $\mathfrak{b}_n^0$ is defined similarly. Write $\mathfrak{a}_n^1$ (resp., $\mathfrak{b}_n^1$) for the annihilator of $\mathfrak{b}_n^0$ (resp., $\mathfrak{a}_n^0$) with respect to the perfect pairing $\langle~,~\rangle_n$. Define $\mathfrak{a}_n'$ to be the image of $\mathfrak{a}_n^1$ under the natural quotient map $\mathfrak{a}_n\lra \mathfrak{a}_n/\mathfrak{a}_n^0$. The module $\mathfrak{b}_n'$ is defined similarly. One then sets $\mathfrak{a}^i= \plim_n\mathfrak{a}^i_n$, $\mathfrak{b}^i= \plim_n\mathfrak{b}^i_n$ ($i=0,1)$, $\mathfrak{a}'= \plim_n\mathfrak{a}'_n$ and $\mathfrak{b}'= \plim_n\mathfrak{b}'_n$, where the transition maps are induced by $c^n_m$. The connections between these modules are recorded in the following lemma.

\bl \label{ses LLTT}
The following statements are valid.
\begin{enumerate}
\item[$(a)$]
We have isomorphisms $\ilim_n \mathfrak{b}_n/\mathfrak{b}^0_n\cong (\mathfrak{a}^1)^{\vee}$ and $\ilim_n \mathfrak{a}_n/\mathfrak{a}^0_n\cong (\mathfrak{b}^1)^{\vee}$. Here $(~~)^{\vee}$ is the Pontryagin dual.
\item[$(b)$] There are short exact sequence of $\La$-modules
\[ 0\lra \mathfrak{a}^0 \lra \mathfrak{a} \lra \mathfrak{a}'\lra 0,\]
\[ 0\lra \mathfrak{b}^0 \lra \mathfrak{b} \lra \mathfrak{b}'\lra 0.\]
\end{enumerate}\el

\bpf
(a) Since $\mathfrak{a}_n^1$ is the annihilator of $\mathfrak{b}_n^0$ with respect to the perfect pairing $\langle~,~\rangle_n$, we have $\mathfrak{b}_n/\mathfrak{b}^0_n\cong (\mathfrak{a}_n^1)^{\vee}$. Taking direct limit and noting ($\Ga$-4), we obtain the first isomorphism. The second isomorphism can be proved similarly.

(b) It suffices to establish the first short exact sequence, the second being similar. Firstly, note that we have the following short exact sequences
\[ 0\lra \mathfrak{a}^0_n \lra \mathfrak{a}^1_n + \mathfrak{a}^0_n \lra \mathfrak{a}'_n\lra 0,\]
\[ 0\lra \mathfrak{b}^0_n \lra \mathfrak{b}_n \lra \mathfrak{b}_n/\mathfrak{b}^0_n\lra 0\]
of $\La$-modules.
Since, one clearly has $\ilim_n\mathfrak{b}_n^0 =0$, it follows from the second exact sequence that $\ilim \mathfrak{b}_n \cong \ilim \mathfrak{b}_n/\mathfrak{b}^0_n$. Upon taking dual and noting (a), we obtain the identification $\mathfrak{a} \cong \mathfrak{a}^1$. It follows from this that upon taking the inverse limit of the first exact sequence, we have
\[ 0\lra \mathfrak{a}^0 \lra \mathfrak{a} \lra \mathfrak{a}'\lra 0,\]
and this completes the proof of the lemma. \epf

The significance of the modules $\mathfrak{a}'$ and $\mathfrak{b}'$ lies in the following result which will be used in the proof of our main
theorem.

\bp[Lai-Longhi-Tan-Trihan] \label{LLTT}
Let $\mathfrak{A}$ be a $\Ga$-system. Then we have
\[ [\mathfrak{a}']^{\iota}_{ns} = [\mathfrak{b}']_{ns}.\]
\ep

\bpf
See \cite[Corollary 3.3.4]{LLTT}.
\epf

\subsection{Finishing the proof}

We return to the setting in Section \ref{Selmer}. In particular, we retain the notation from there. Recall that $S_{div}(E/F_n)$ denotes the $p$-divisible part of $\Sel^{\s,str}(E/F_n)$. We set $\mathfrak{a}_n(E) = \mathfrak{b}_n(E) = \Sel^{\s,str}(E/F_n)/S_{div}(E/F_n)$. Write $S_{div}(E/F_\infty) = \ilim_n \Sel_{div}(E/F_n)$. The limit of these groups sit in the following short exact sequence
\[ 0\lra S_{div}(E/F_\infty) \lra \Sel^{\s}(E/F_\infty) \lra \ilim_n \mathfrak{a}_n(E)\lra 0, \]
where we have identified $\Sel^{\s,str}(E/F_\infty) \cong \Sel^{\s}(E/F_\infty)$ (cf. Lemma \ref{strict = usual}).

The morphisms $r^n_m(E)$ are induced by the restriction maps on the signed Selmer groups which in turn are induced by those from cohomology groups. On the other hand, the morphisms $c^n_m(E)$ are induced by the corestriction maps on the signed Selmer groups which in turn are induced by those from cohomology groups.
As seen in Proposition \ref{Flach}, there is a perfect pairing
\[ \Sel^{\s,str}(E/F_n)/S_{div}(E/F_n)\times \Sel^{\s,str}(E/F_n)/S_{div}(E/F_n) \lra \Qp/\Zp,\]
which we shall denote by $\langle~,~\rangle_{E,n}$.

\bl \label{signed Ga-system}
Suppose that $X^{\s}(E/F_\infty)$ is torsion over $\Zp\ps{\Ga}$. Then
\[ \mathfrak{A} = \big\{ \mathfrak{a}_n(E), \mathfrak{b}_{n}(E), \langle~,~\rangle_{E,n}, r^n_m(E), c^n_m(E) ~\big|~ n\geq m \geq 0 \big\}\]
is a $\Ga$-system.
\el

\bpf
The verification of ($\Ga$-1)-($\Ga$-4) is a straightforward (but maybe tedious) exercise.
Finally, by definition, we have that
$\mathfrak{a}(E)$ ($=\mathfrak{b}(E)$) is a $\La$-submodule of $X^{\s}(E/F_\infty)$. Since $X^{\s}(E/F_\infty)$ is assumed to be torsion over $\Zp\ps{\Ga}$, so is $\mathfrak{a}(E)$. Thus, we have the lemma.
\epf

The next lemma is concerned with the structure of $Y(E/F_\infty): = \Sel_{div}(E/F_\infty)^{\vee}$. 

\bl \label{Y structure}
Suppose that $X^{\s}(E/F_\infty)$ is torsion over $\Zp\ps{\Ga}$.
 Then we have $[Y(E/F_\infty)]= [Y(E/F_\infty)]_{si}$.
\el

\bpf
By \cite[Theorem 4.1.3]{LLTT}, we see that $\big(\Sel^{\s}(E/F_\infty)^{\Ga_n}\big)_{div}$ is annihilated by a product $f$ of simple elements for every $n$. Now by Lemma \ref{control kernel}, $S_{div}(E/F_n)$ can be viewed as a submodule of $\big(\Sel^{\s}(E/F_\infty)^{\Ga_n}\big)_{div}$. Thus, it is also annihilated by $f$ for every $n$. Consequently, the module $Y(E/F_\infty)$ is annihilated by $f$, and in particular, one has $[Y(E/F_\infty)]= [Y(E/F_\infty)]_{si}$.
\epf
We can now give the proof of the main result of the paper.

\bpf[Proof of Theorem \ref{main theorem}]
As seen in Subsection \ref{first reduction}, it remains to show that
\[ [X^{\s}(E/F_\infty)]^{\iota}_{ns} = [X^{\s}(E/F_\infty)]_{ns}. \]
On the other hand, by a combination of Proposition \ref{LLTT} and Lemma \ref{signed Ga-system}, we obtain $[\mathfrak{a}'(E)]_{ns}^{\iota} = [\mathfrak{b}'(E)]_{ns}$. Since $\mathfrak{a}'(E)= \mathfrak{b}'(E)$, the conclusion of the theorem will follow from this once we show that $[\mathfrak{a}'(E)]_{ns} = [X^{\s}(E/F_\infty)]_{ns}$.
As a start, consider the following commutative diagram
\[
 \entrymodifiers={!! <0pt, .8ex>+} \SelectTips{eu}{}\xymatrix{
    0 \ar[r]^{} & S_{div}(E/F_n)  \ar[d] \ar[r] &  \Sel^{\s,str}(E/F_n)
    \ar[d] \ar[r] & \mathfrak{a}_n(E)\ar[d] \ar[r]&0 \\
    0 \ar[r]^{} & S_{div}(E/F_\infty)^{\Ga_n} \ar[r]^{} & \Sel^{\s}(E/F_\infty)^{\Ga_n} \ar[r] &
    \Big(\ilim_m \mathfrak{a}_m(E)\Big)^{\Ga_n} &  }
    \]
with exact rows, where we write $\Ga_n=\Ga^{p^n}$. By definition, the kernel of the rightmost map is $\mathfrak{a}^0_n(E)$. Since the middle map is injective by Lemma \ref{control kernel}, it follows from the snake lemma that $\mathfrak{a}^0_n(E)$ injects into $S_{div}(E/F_\infty)^{\Ga_n}/S_{div}(E/F_n)$. Since $S_{div}(E/F_\infty)^{\Ga_n}$ is annihilated by a product $f$ of simple elements for every $n$ by Lemma \ref{Y structure}, so is $\mathfrak{a}^0_n(E)$. Hence we have $[\mathfrak{a}^0(E)]= [\mathfrak{a}^0(E)]_{si}$. It then follows from this and Lemma \ref{ses LLTT} that
$[\mathfrak{a}(E)]_{ns} = [\mathfrak{a}'(E)]_{ns}$.

On the other hand, we also have the following short exact sequence
\[ 0 \lra \mathfrak{a}(E)\lra X^{\s}(E/F_\infty) \lra Y(E/F_\infty)\lra 0.\] 
It then follows from this and Lemma \ref{Y structure} that $[\mathfrak{a}(E)]_{ns}=[X^{\s}(E/F_\infty)]_{ns}$. Combining this with the conclusion obtained in the previous paragraph, we have $[\mathfrak{a}'(E)]_{ns} = [X^{\s}(E/F_\infty)]_{ns}$ as required.
\epf

\end{document}